\documentclass[11pt]{amsart}
\usepackage{latexsym}
\usepackage{graphics}
\usepackage{graphicx}
\usepackage{epstopdf}
\usepackage{tikz}
\usepackage{amsmath}
\usepackage{float}
\usepackage[normalem]{ulem}

\usepackage{caption}
\usepackage[labelformat=simple]{subcaption}

\newtheorem{theorem}{Theorem}[section]

\newtheorem{lemma}[theorem]{Lemma}

\theoremstyle{remark}
\newtheorem{remark}[theorem]{Remark}

\title{Alternating links have at most polynomially many Seifert surfaces of fixed genus}
\author{Joel Hass, Abigail Thompson, Anastasiia Tsvietkova}
\date{}
\subjclass[2010]{}
\everymath{\displaystyle}
\begin{document}

 \footnotesize
 \begin{abstract}
 Let $L$ be a non-split prime alternating link with $n>0$ crossings.
 We show that for each fixed $g$, the number of genus-$g$ Seifert surfaces for
 $L$ is bounded by an explicitly given polynomial in $n$. The result also holds for all spanning surfaces of fixed Euler characteristic.
 Previously known bounds were exponential.
 \end{abstract}

\maketitle
\normalsize

\section{Introduction}

A \textit{Seifert surface} for a link in $S^3$ is a connected orientable surface, embedded in $S^3-L$, whose boundary is isotopic to the link.
If the orientability condition is omitted, the surface is called a \textit{spanning surface}.
Let $L$ be a non-split prime alternating link with an $n$-crossing diagram, where $n>0$. In this paper we give an upper
bound on the number of isotopy classes of spanning surfaces of $L$ that have a fixed Euler characteristic.
The bound is given by an explicit polynomial in $n$. Our methods apply to  Seifert surfaces, spanning surfaces and to more general essential surfaces
with non-meridional boundary.  However in this paper we focus our attention on Seifert surfaces and state our results
mainly for that class of surface. We show that the number of genus-$g$ Seifert surfaces in the complement of an $n$-crossing alternating diagram is at most $ (4n)^{   64g^2 -48g }$.

In general, a knot complement can contain many, and in some cases infinitely many, non-isotopic Seifert surfaces of a given genus (see, for example, \cite{Eisner}). But
in a hyperbolic manifold the number of spanning surfaces is always finite.
This can be seen by homotoping each surface to a least area representative and applying
 the Gauss-Bonnet Theorem and Schoen's curvature estimates \cite{Hass, Schoen}. Alternately the surfaces can
 be homotoped to pleated surfaces as in \cite{Thurston}.
 This type of argument applies also to $\pi_1$-injective immersions, but
 is not constructive and gives no explicit bound on the number of surfaces of a given genus.
Normal surface theory can also be used to bound the number of spanning surfaces of genus $g$.
Each surface can be isotoped to be normal, and can then be expressed as a sum  of
finitely many fundamental normal surfaces. However this process leads to an exponential bound on the number
 of spanning surfaces of genus $g$
as a function of $n$.  This is due to the exponential growth of the number of fundamental surfaces of a given genus, and even
of the number of vertex fundamental surfaces, as a function of the number of tetrahedra $t$ \cite{HassLagariasPippenger}.  An additional difficulty is that an incompressible surface may not be fundamental, so that
 one must also count combinations of lower genus fundamental surfaces that combine to form a given genus \cite{Haken:61}.

 The surfaces we consider are  embedded and incompressible, but not necessarily disjoint. The number of disjoint incompressible surfaces in a manifold
 is easier to bound. This was first observed by Kneser \cite{Kneser:29} for  closed incompressible surfaces.
 Kneser showed that the number of such surfaces is bounded by a linear function of the number of tetrahedra $t$ required to triangulate the manifold.  For a link complement,
 $t$ is a linear multiple of the number of crossings in the link diagram. Hence the number of disjoint spanning surfaces realizable for a link $L$
 grows linearly with $n$  \cite{Hempel}. Kneser's arguments also apply to surfaces with boundary in link complements.

If we fix a link complement, our results additionally yield exponential upper bound in terms of genus (rather than crossing number). This can be compared with the results of Masters \cite{Masters} and Kahn-Markovic \cite{KahnMarkovic} for essential immersed closed surfaces in a closed manifold, where the manifold is fixed and it is shown that the number of surfaces grows exponentially with the genus.

A related  problem for closed surfaces was studied by the authors in \cite{HTT}, where it was proved that the number of closed incompressible genus-$g$ surfaces in a prime alternating link complement is bounded by $C_g n^{40g^2}$, where $C_g$ is an explicit constant depending only on the genus $g$, and $n$ is the number of crossings. The proof for Seifert surfaces and spanning surfaces presented here needs to consider cases that cause difficulties not encountered with closed surfaces. In \cite{HTT} a surface is put in standard position with respect to the projection plane and the link diagram, and then decomposed into disks bounded by polygons and lying above or below the plane. By summing the contributions  to the Euler characteristic of the surface of each region, a bound on the number  and complexity of the region is obtained. Each region makes a negative contribution to the  Euler Characteristic, bounding their number, and the possible ways for polygons to appear on a link diagram can then be analyzed, providing an upper bound for the surface count. The standard position and Euler characteristic arguments can be extended to Seifert surfaces, but the resulting regions for surfaces with boundary  include cases that contribute zero to the Euler characteristic computation.  This leads to an exponential explosion in the number of possible intersection configurations relative to the number of crossings. However many of these configurations give rise to isotopic surfaces, and we
can  show that the number of surfaces up to isotopy  still grows polynomially with the number of crossings. This is carried out in Sections~4 and 5. \color{black}

\section{Standard position for Seifert surfaces}

A standard position for surfaces in an alternating knot complement was introduced by Menasco for closed surfaces and for surfaces with meridianal boundary \cite{Menasco1984}, and extended by Menasco and Thistlethwaite to general surfaces with  boundary \cite{Menasco1992}.
We briefly review these techniques, with some minor modifications to the arguments.

 A reduced alternating diagram $D$ of a prime alternating link $L$ can be placed in a projection plane $Q$ except for two small arcs near each crossing,
 one of which drops below $Q$, and one of which rises above it.
 $L$ then lies on a union of two overlapping 2-spheres in $S^3$, $S_+$ and $S_-$, which agree with $Q$
 except along small balls around each crossing, called \textit{bubbles}.  The spheres $S_+$ and $S_-$ go over the top and bottom hemispheres of each bubble, respectively.
We denote by $B_+$ the ball  in $S^3$ lying above the projection plane and bounded by $S_+$,  and by $B_-$  the ball  lying below the projection plane and bounded by $S_-$.

Suppose $F$ is a spanning surface for $L$. It is shown in Proposition~2.1 of  \cite{Menasco1992} that $F$ can be isotoped rel boundary so that it intersects $Q$ transversally except in two situations:
1) $F$ meets a bubble  in a saddle near a crossing, as in Figure~\ref{saddle};
2) At finitely many points, $F$ twists around a strand of $L$ as in Figure~\ref{Twist}.

$F$ is then said to be in the \textit{standard position}. In (2), we call the arc of intersection of $F\cap (Q-D)$, together with its endpoints on $D$, an \textit{interior arc}. An interior arc can run along $Q$ to connect two points on $D$, or between a point on $D$ and a saddle. An arc of $F\cap Q$ that coincides with an arc of $D$ and lies between two interior arcs is called a \textit{boundary arc}.

\begin{figure} [htbp]
\centering
\begin{subfigure}[b]{0.45 \textwidth}
\centering
\includegraphics[scale=0.29]{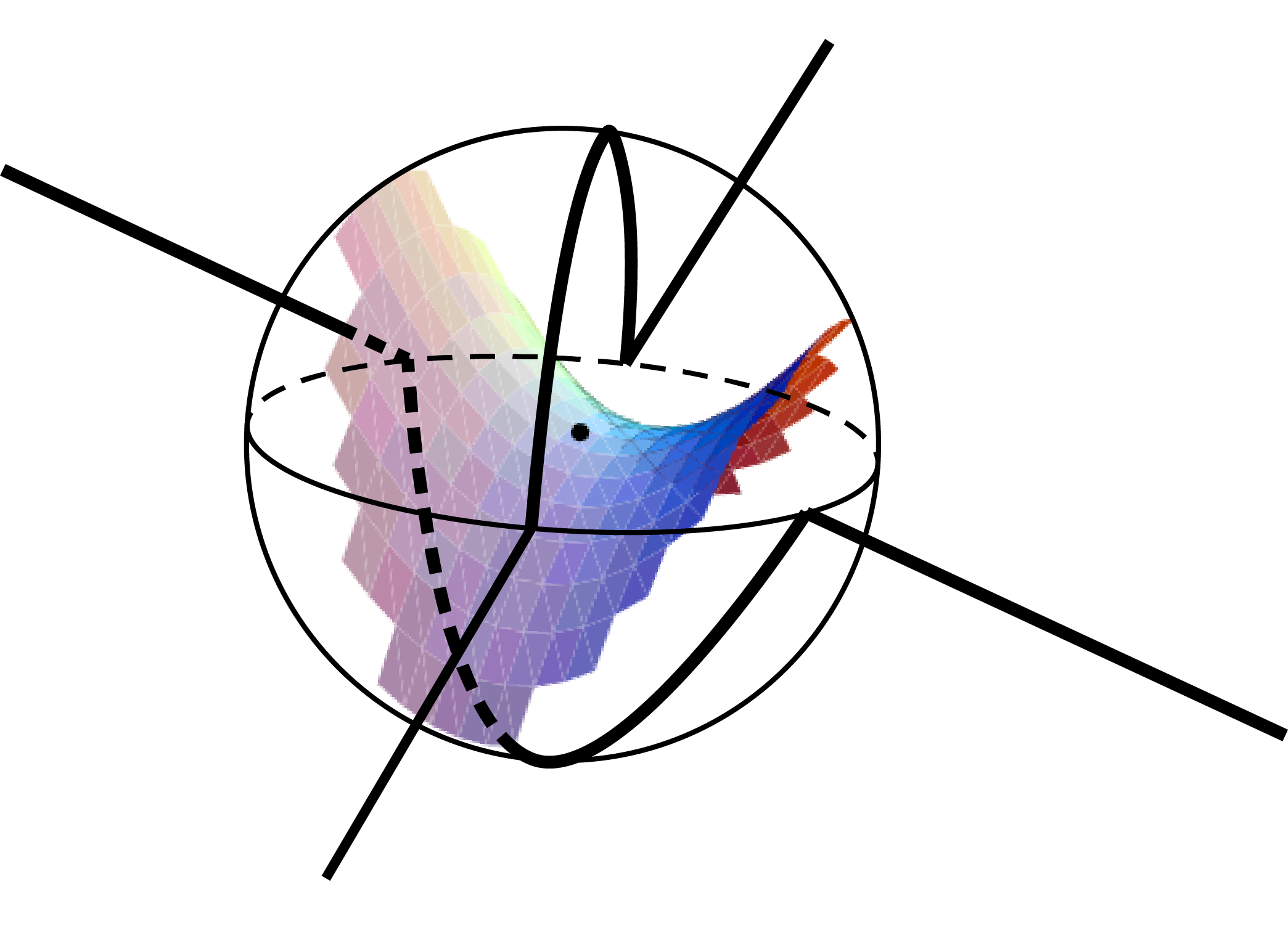}
\caption{$F$ is tangent to the projection plane at a saddle point contained in a bubble.}
\label{saddle}
\end{subfigure}
\begin{subfigure}[b]{0.45 \textwidth}
\centering
\includegraphics[scale=1.1]{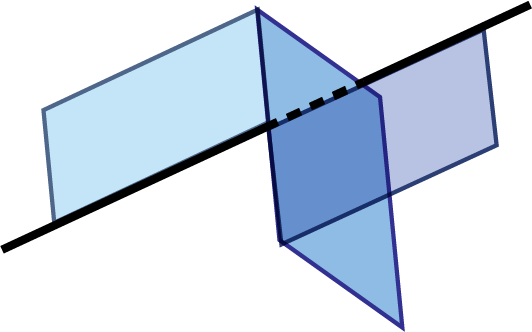}
\caption{An interior arc of $F \cap Q$ ends at a point where $F$ twists from being above  to below $Q$.}
\label{Twist}
\end{subfigure}
\caption{}
\end{figure}

$F$ is divided into regions in $B_+$ and $B_-$ by closed curves in $F  \cap S_+$ or  $F \cap  S_-$. Parts of these curves lie in the boundary of $F$ (\textit{i.e.} run along  the link). We assign to each closed curve  $C$ in $F  \cap S_+$ or  $F \cap  S_-$  a word in the letters $B$ and $S$, defined up to cyclic order, as follows. Orient  $C$ and  starting from an arbitrary non-crossing point, follow $C$ until returning to that point. For every saddle   passed on the way, add  an $S$ to the word and for every point that passes between interior and boundary arcs, add a $B$. Figure~\ref{Figure1.2}  gives an example of a link, and a curve of  $F  \cap S_+$ that gives the word $BBSSS$. The link is depicted in black, the curves of $F  \cap S_+$ with dashes, and the curves of $F  \cap S_-$ with dots.

Define the complexity of a surface $F$ in standard position to be a pair $(s,c)$ , where
$s$ is the sum of the number of $S$'s associated to all curves in $F \cap S_+$ and $F \cap S_-$, and $c$ is the number of such curves. If $F$ minimizes this complexity in lexicographic order among all standard position
surfaces in its isotopy class, then $F$ is said to have {\em minimal complexity}.

We refer to a segment of the link diagram $D$ that travels between two successive crossings of $L$ and lies in the projection plane $Q$ as an \textit{edge} of $D$, and use \textit{arc} to refer to subcurves of $F  \cap S_+$ and  $F \cap  S_-$ running between $S$'s and $B$'s.

The following lemma summarizes the properties of spanning surfaces in standard position that follow from the work of Menasco and Thistlethwaite \cite{Menasco1992}.

\begin{lemma}\label{properties}
Suppose a spanning surface $F$ has minimal complexity.
Then the curves of $F\cap S_+$ and  $F\cap S_-$ and the associated words in the letters $B$, $S$ have the following properties:
 \begin{enumerate}
\item \label{Disks}  The curves of $F \cap S_+$ and $F\cap S_-$ subdivide $F$ into disks, each disk lying either in $F \cap B_+$ or $F \cap B_-$.
\item \label{SaddleTwice}  No curve passes through the same bubble twice.
\item \label{EveryB} Every edge of the diagram $D$ meets an interior arc of $F \cap Q$.
 \item \label{ArcTwice}  No curve contains two interior arcs with endpoints on an edge $A$ and lying on the same side of $A$.
  \item \label{Balance} An equal number of curves of $F\cap S_+$ pass through each side of a saddle. The same holds for $F\cap S_-$.
 \item \label{SaddleArc} No curve passes through a saddle and then meets an edge of $D$ adjacent to the saddle.
 \item \label{NoEmpty} No word is empty.

 \item \label{ConsecutiveB's} Letters $B$ in a word appear in consecutive pairs.
 \item \label{NotJustS}
 There is no word consisting entirely of $S$'s.

 \item \label{Length4} Each word has length at least four.

 \end{enumerate}
 \end{lemma}

\begin{proof}
We indicate the proofs for $F\cap S_+$ below. The same arguments apply to $F\cap S_-$.

For (\ref{Disks}), see Proposition~2.2(i) in \cite{Menasco1992}. For (\ref{SaddleTwice}) and (\ref{ArcTwice}), see Proposition~2.2(ii), and for (\ref{SaddleArc}) see Proposition~2.3.

The proof of (\ref{EveryB}) follows from the fact that $L$ is alternating and is the boundary of $F$. Hence, every edge is adjacent to an overpass and underpass, and a part of the surface near an edge changes from being in $B_+$ to being in $B_-$ somewhere between the two adjacent crossings. This gives rise to an interior arc meeting the edge.
Claim (\ref{Balance}) follows from the fact that each saddle of a surface results
in one intersection curve on each side of a crossing. For (\ref{NoEmpty}), if there is a component of $F\cap S_+$ with no saddles or punctures, take an
innermost such. It bounds a disk in $B_+$, so we can isotop $F$ to eliminate the curve of intersection with $S_+$ and reduce the complexity.
For (\ref{ConsecutiveB's}), note that each boundary arc in a curve contributes two successive $B$'s.
Claim (\ref{NotJustS}) is proved in Lemma 2 of \cite{Menasco1984} for closed surfaces. The proof for spanning surfaces is exactly the same.

\begin{figure} [htbp]
\centering
\includegraphics[scale=0.4]{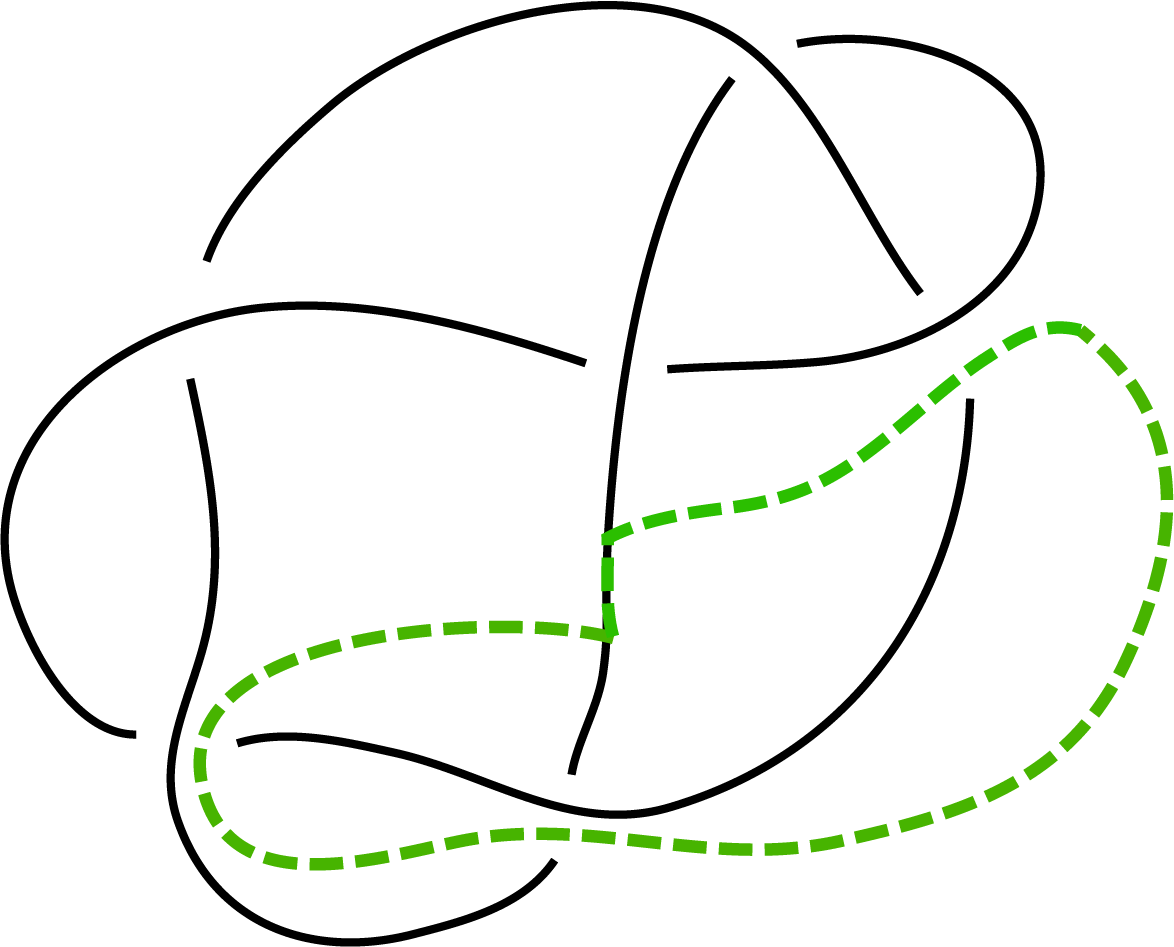}
\caption{The link $L$ and a $BBSSS$ curve from $F\cap S_+$}\label{Figure1.2}
\end{figure}

 For (\ref{Length4}), if the word for a curve has just two letters, it is one of $BB$, $SB$ or $SS$. The word $SS$ contradicts
(\ref{SaddleTwice}), $SB$ contradicts (\ref{ConsecutiveB's}), and $BB$ contradicts  (\ref{ArcTwice}). Among 3-letter words, we have ruled out $SSS$. A curve of type  $BBS$ can be perturbed so that it intersects the link $L$ at most three times.  Thus the perturbed curve intersects twice, implying either that $D$ is not prime or that Case (6)  is violated. Hence the length of the word is at least four.
\end{proof}

\section{Decomposing a surface into regions}

We henceforth consider a surface $F$ in standard position with  minimal complexity. In this section, we decompose the surface into polygonal faces, determined by the intersections of $F$ with $Q$, and analyze the contributions of the faces to the Euler characteristic of the surface. A similar technique was used in \cite{HTT} for closed surfaces, and  in \cite{Menasco1992} for a different purpose.

We decompose $F$ using the arcs of $F \cap Q$. These form a
graph on $F$. At each saddle four polygonal faces and four arcs of $F \cap Q$ meet at the  saddle point, where we add a vertex to the graph. We also add a vertex at every intersection of $D$ and an interior arc of $F \cap Q$. The resulting graph $\Gamma$ has vertices of valence four at the centers of the saddles and vertices of valence three at endpoints of interior arcs that meet $D$. The graph $\Gamma$ cuts $F$ into a collection of disks by Lemma \ref{properties} (\ref{Disks}) that we call polygonal faces or regions. The vertices and edges of a region are the vertices and edges of $\Gamma$ respectively.

 \begin{lemma}\label{Polygons}
The Euler characteristic of $F$ can be computed by adding the contribution of each region. A region $E$ contributes $1- {s_0}/{4}-{b_0}/{4}$ to $\chi(F)$, where
  $s_0$ is the number of $S$'s in the word associated to the boundary of $E$, and  $b_0$ is the number of $B$'s.
\end{lemma}
 \begin{proof}
 Enumerate all curves $C_i, i=1, \dots  r$, in $F \cap S_+$ and $F \cap S_-$. Suppose $C_i$ is the boundary of a region $E_i$ of $F$ with interior in
 $B_+$ or  $B_-$.
 The Euler characteristic of $F$ can be recovered by summing the contributions of each of these regions $E_i, i=1, \dots , r$.

 Four distinct regions share a vertex at a saddle, and two regions share a vertex at the endpoint of an interior arc that meets $D$. Two regions share an edge of $\Gamma$ that coincides with an interior arc, and there is just one region at every edge of $\Gamma$ that corresponds to a boundary arc. Hence the Euler characteristic  $\chi(F) = v-e+f$ can be  distributed among vertices and edges as follows.

 For vertices, $+1/4$ is  allocated to each
 vertex of a region on a saddle (\textit{i.e.} $S$ contributes +1/4 as a vertex), +1/2 to each vertex at the end of an interior arc that meets $D$ (\textit{i.e.} $B$ contributes +1/2 as a vertex). For edges,  the contribution is -1/2 for an edge of $\Gamma$ that corresponds to an interior arc, and -1 to an edge   that corresponds to a boundary arc.

 Now let's distribute the contributions of the edges of $\Gamma$ between their vertices. Every $B$ is an endpoint of one interior arc and one boundary arc. Every $S$ is an endpoint of two interior arcs. In an interior arc with the contribution of -1/2, we can view the contribution of $B$ as -1/4 and of $S$ as -1/4.  In a boundary arc with the contribution -1, the contribution of every $B$ is -1/2. Thus the contribution of a region with $s_0$ vertices at saddles and $b_0$ vertices at the ends of interior arcs that meet $D$ is
 $$
 s_0(1/4-1/4-1/4) +b_0(1/2-1/4-1/2) + 1 = 1 - s_0/4-b_0/4.
$$
  \end{proof}

\begin{remark}\label{Contribution}
It follows from Lemma \ref{properties} (\ref{Length4}) that all curves of $F \cap S_+$ or $F \cap S_-$ make a negative contribution to the Euler characteristic of a minimal complexity spanning surface, no greater than -1/4, except for $BBBB$ and $BBSS$ curves, which contribute zero. We analyze $BBBB$ curves in the next section.
\end{remark}

\section{Collections of $BBBB$ curves}

We say that a collection of  regions in $F$ is {\em connected} if the  dual graph,
formed by taking a vertex for each region and an edge for two regions that share a  common
 arc,  is connected. A  connected collection of $BBBB$ regions is {\em maximal}  if it is not a
strict subset of a connected collection of $BBBB$ regions.

\begin{figure}[H]
\centering
 \includegraphics[scale=0.7]{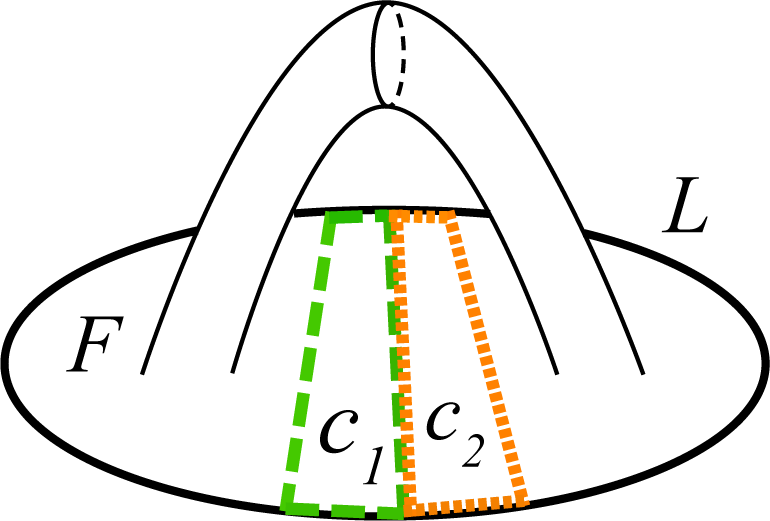}
\caption{A connected collection of two regions, each bounded by a $BBBB$ curve.}\label{Parallel}
 \end{figure}

 \begin{lemma}\label{Regions}
 Assume that $L$ is not a $(2,n)$-torus link.
 A connected collection of  $BBBB$ regions form a subsurface of $F$ with interior homeomorphic to a disk.
 The set of all maximal connected collections of $BBBB$ regions forms a collection of disks
 with no pair of disks sharing a common arc in their boundaries.
  \end{lemma}
 \begin{proof}
We first use induction to show that a connected collection of  $BBBB$ regions $C$ has interior homeomorphic to a disk.
If $C$ consists of a single $BBBB$ region, then it is
a disk, since any single region is a disk.  Now consider the case when $C$ consists of the union of
$k$ distinct $BBBB$ regions.   The dual graph of $C$ is a finite connected graph, so we can remove some region $R$
to get a collection of $k-1$ $BBBB$ regions $C_0$ which is also connected.  By induction $C_0$ is
a disk, and $C$ is obtained from $C_0$ by adding a single $BBBB$ region $R$ sharing at least  one arc with $C_0$.  Now
$C_0 \cap R$ consists of either one or two interior $BB$ arcs.  If one, then since the union
of two disks intersecting along a proper arc on their boundary is a disk, it follows that the interior of $C$ is also a disk.
If $C_0 \cap R$ has two components, then $C$ is an annulus or Mobius band, properly embedded in $F$, and $\partial C$  coincides with
either one or two components of $L$.  The boundary of a regular neighborhood in $S^3$ of this annulus
or Mobius band is a torus,
and since $L$ is alternating, the torus must be compressible in the complement of $L$ \cite{Menasco1984}, and hence unknotted
and the boundary of a solid torus in the complement of $L$.
Since each curve of $\partial C$  is isotopic to a curve on this torus  that intersects its meridian twice, and $L$ is not split, $L$ must be a $(2,n)$-torus link, contradicting our hypothesis.
We conclude that the
connected collection of  $BBBB$ regions $C$  is homeomorphic to a disk.

A maximal connected collection of $BBBB$ regions cannot share an interior arc with
a another such region, since if it did neither would be maximal.  Thus maximal connected collections
are disks with no common arc in their boundaries.
 \end{proof}

A consequence of Lemma~\ref{Regions} is that the isotopy class of a spanning surface does not depend at all on the location of
$BBBB$ curves.

\begin{lemma}\label{BS}
The curves  of $F\cap{S_+}$ and $F\cap{S_-}$ that are not of type $BBBB$, together with the link $L$, determine a unique spanning surface $F$, up to isotopy.
\end{lemma}

\begin{proof}
  By Lemma \ref{Regions}, maximal connected collections of $BBBB$  regions form disjoint disks when $L$ is not a $(2,n)$-torus link.
  The image of the surface in the complement of the disks
is determined by the configuration of curves  that are not $BBBB$ curves.
Two Seifert surfaces that agree in the complement of a collection of disks are isotopic, since the link $L$ is not split, and therefore has
irreducible complement.  Thus the  isotopy class of the spanning surface $F$ is determined once the curves that are not $BBBB$
are specified.  Finally, we note that a $(2,n)$-torus link has a unique incompressible spanning surface.
\end{proof}

\section{The number of Seifert surfaces of fixed genus}

In this section, we bound the number of curves in $F\cap{S_+}$ and $F\cap{S_-}$, and the maximum length of the word associated to each curve. This in turn gives an upper bound for the number of Seifert surfaces of a fixed genus, up to isotopy. Here and further we assume that genus $g>0$.

Let $C_1$ denote the set of curves of $F\cap{S_+}$ and $F\cap{S_-}$ that are not of type $BBBB$  or $BBSS$,
and $C_2$   the set of  $BBSS$ curves,

\begin{lemma}\label{C1} $|C_1| \le 8g -4$.
\end{lemma}
\begin{proof}
 By Remark \ref{Contribution},  every curve in $C_1$ gives a contribution to the absolute value of the Euler characteristic of $F$ of at least 1/4. The Euler characteristic of $F$ is $1-2g$.
Thus the maximal possible size of $|C_1|$ is  $1-2g/(-1/4) = 8g-4$.
\end{proof}

\begin{lemma}\label{WordsLength} The length of the word associated to any  curve in $(F\cap S_\pm)$ is at most $8g-4$.
\end{lemma}
\begin{proof} We consider words  associated to the curves  in $C_1$, since all other words are of length 4.
It follows from Lemma \ref{Polygons} that each word of $C_1$ gives a negative contribution to the Euler characteristic of $F$ .
The Euler characteristic of $F$ is $1-2g$, and each word contributes $1- {s_0}/{4}-{b_0}/{4}$ by Lemma \ref{Polygons},
where $b_0$ is the number of $B$'s, and $s_0$ is the number of $S$'s in the word.
Thus the longest possible word has length $b_0+s_0=8g -4$.
\end{proof}

\begin{lemma}\label{C2}
$|C_2| \le 4g-4 $.
\end{lemma}
\begin{proof}
  Every $BBSS$ curve contains two saddles. Theorem 3 of \cite{Menasco1992} gives an upper bound of
 $  (-\chi(F)-  b)  =  2g-2 $  for the number of saddles in a minimal complexity diagram, where $b$ is the number of
 boundary components of $F$.
Each   $BBSS$ curve meets two saddles and each saddle meets four curves, thus
at most  $ 4g-4 $ $BBSS$ words are associated to the surface $F$. So $|C_2| \le   2( 2g-2)  = 4g-4 $.
\end{proof}

\begin{theorem}\label{NumberOfSpanning}
For a prime non-split alternating link $L$ with $n$ crossings, the number of isotopy classes of genus-$g$ Seifert surfaces is at most $ (4n)^{  64g^2 -48g}$ .
\end{theorem}

\begin{proof}
 In each  isotopy class we choose a surface that is in minimizing standard position.
There are at most 8g -4 curves in $C_1 $   by Lemma~\ref{C1} and $4g-4$ curves in $C_2$   by Lemma ~\ref{C2}.
We count the number of possible configurations these curves can realize, up to isotopy of the curves.

Each of the  curves  in $C_1 $ has an associated word with length at most $8g-4$ by Lemma \ref{WordsLength}.
Consider a curve of $F \cap S_+$. The link $L$ has $n$ crossings, and each crossing gives rise to two choices for the location of an $S$
adjacent to that crossing. There are $2n$ edges in the link diagram, each edge having two sides,
giving  $4n$ choices for where an interior arc meets an edge.
Choosing  successive sides of  saddles and  edges determines a curve up to isotopy.
Therefore the number of isotopy classes of a curve in $C_1$, each such curve having length at most $8g-4$ is less than $(4n)^{8g-4}$.
For the entire collection of up to  $ 8g -4$ curves, the total number of isotopy classes is bounded by $(4n)^{{(8g-4)(8g-4)}}$.
Similarly, the number of configurations for curves in $C_2$ is bounded by $(4n)^{{(4)(4g-4)}}$.
Once the curves in $C_1$ and $C_2$  are fixed,  Lemma \ref{BS} shows that the spanning surface is determined up
to isotopy. This gives an upper bound of  $ (4n)^{  64g^2 -48g }$  possible  isotopy classes for $F$.
\end{proof}

\begin{remark}
We have not used the orientability of $F$ in the arguments above, except when we replace Euler characteristic
with genus. Therefore, a similar upper bound holds for the number of spanning surfaces, oriented or not, if genus $g$ is replaced by $(1-\chi(F))/2$.
\end{remark}

\begin{remark} The bound in Theorem \ref{NumberOfSpanning} is polynomial in $n$ when the genus is fixed.
Thus the number of genus ten Seifert surfaces for a link is bounded by a polynomial function of the number of crossings as $L$ varies.
However for a fixed link $L$,  the number of surfaces can grow exponentially with the genus.
This can be compared with the results on immersed closed surfaces in closed hyperbolic 3-manifolds
\cite{KahnMarkovic}, \cite{Masters}, and with similar results for closed surfaces in alternating links \cite{HTT}.
We note that J. Banks constructed explicit families of prime alternating knots for which the number of minimal genus Seifert surfaces  grows
 exponentially with the genus  (\cite{Banks}).
\end{remark}

\begin{remark}
The Kakimizu complex of a link $L$ is a simplicial complex whose vertices correspond to isotopy classes of minimal genus Seifert surfaces of $L$, and edges correspond to disjoint (up to isotopy) pairs of such surfaces (\cite{Kakimizu}).  Theorem \ref{NumberOfSpanning} gives an upper bound on the number of vertices in a connected component of this complex.
\end{remark}

 \section{Acknowledgments}
 J. Hass was partially supported by NSF grant DMS-1758107.
A. Thompson was partially supported by NSF grant DMS-1664587.
A. Tsvietkova was partially supported by NSF grants DMS-1406588, DMS-1664425, and by Okinawa Institute of Science and Technology funding.

Joel Hass \\
Department of Mathematics\\
University of California, Davis \\
One Shields Ave, Davis, CA 95616 \\
hass@math.ucdavis.edu

Abigail Thompson \\
Department of Mathematics\\
University of California, Davis \\
One Shields Ave, Davis, CA 95616 \\
thompson@math.ucdavis.edu

Anastasiia Tsvietkova\\
Department of Mathematics and Computer Science\\
Rutgers University-Newark \\
101 Warren Street, Newark, NJ  07102\\
a.tsviet@rutgers.edu

\end{document}